\documentclass[12pt,twoside]{amsart}
\usepackage{amsmath}
\usepackage{amssymb}
\usepackage{amscd}
\usepackage{xypic}
\xyoption{all}
\title[Nilpotent orbit theorem in $p$-adic Hodge theory]{Nilpotent orbit theorem in $p$-adic Hodge theory}

\author{Mohammad Reza Rahmati, Gerardo Flores}
\thanks{}
\address{.
\hfill\break 
\hfill\break \\
\hfill\break }
\email{mrahmati@cimat.mx, gflores@cio.mx}

\newcommand{\comments}[1]{}

\newtheorem{theorem}{Theorem}[section]
\newtheorem{proposition}[theorem]{Proposition}

\newtheorem{definition}[theorem]{Definition}
\newtheorem{remark}[theorem]{Remark}
\newtheorem{example}[theorem]{Example}


\begin{document}

\begin{abstract}
We state and prove three orbit theorems on the period domains for the $p$-adic Hodge structure analogous to the complex case. We shall consider the variation of de Rham (resp. \'etale) cohomology in a family of projective varieties $f:\mathfrak{X} \to S$ defined over a p-adic field. First, we show that any nilpotent orbit in the period domain of p-adic Hodge structures converges to a semistable point (filtration) in the period domain of the p-adic Hodge structure. Furthermore, the nilpotent orbits of the limit point are asymptotic to the twisted period map [Theorem \ref{thm:nilpotent-orbit}]. The orbit theorems come with some estimates of the distance between the nilpotent orbit and the twisted period map. The distance estimate in the p-adic nilpotent orbit theorem is given concerning the non-archimedean metric and is based on the p-adic Fourier analysis of Amice-Schneider. The result is analogous to the orbit theorems of W. Schmid [\cite{Sch}-1973] on complex Hodge structures. Our proof is based on a \textit{Geometric Invariant Theory} (GIT) criterion for semi-stability (Kempf-Ness theorem) and estimates from the (Amice-Schneider) p-adic Fourier theory. We also state the $SL_2$-orbit theorem in the p-adic case, [Theorem \ref{th:homomorphism}]. Finally, we explain how the nilpotent orbit theorem should be modified and stated for a variation of the mixed Hodge structure [Theorem \ref{thm:mixed-orbit}].
\end{abstract}

\maketitle


\section{Introduction}\label{sec:intr}
The $p$-adic Hodge theory studies the analogy between Galois representations $V$ and certain filtered modules $(M, F^{\bullet})$ with a semilinear action of $\phi$ [the Frobenius] and $\Gamma$ [certain subgroup of the absolute  Galois group], called $(\phi,\Gamma)$-modules. The action is usually equipped with a nilpotent transformation $N$, with transverse property to the Hodge filtration $F^{\bullet}$ also called $(\phi, \Gamma, N)$-module. The module $M$ can be obtained from the representation $V$ by a specific extension of coefficients. Conversely, the representation $V$ can be constructed from $M$ by specialization [taking fixed points under Galois action]. We provide this example in this text; however, the interested reader can find the complete discussion in standard texts of $p$-adic Hodge theory, \cite{B, BP, DOP, FO, I}.

The analog of the complex Hodge decomposition (HS) in p-adic theory is deduced from the convergence of Hodge to de Rham spectral sequence. One obtains a Hodge-type decomposition for the p-adic \'etale cohomology over $\mathbb{C}_p$. The p-adic \'etale cohomology at the same time is also a representation of the absolute Galois group of the ground p-adic field. Besides, both decompositions above can be equally described through the convergence of the Hodge to de Rham spectral sequence.
The asymptotic theory of the p-adic Hodge structure has been attended sufficiently in the literature. One reason is that asymptotic meaning is not as straightforward as in the real or complex variable case. Also, exponentiating a nilpotent transformation or taking the logarithm with a twisted variable to preserve the HS is not as easy as the complex case. 

Nilpotent orbit theorems in asymptotic Hodge theory refer to the two theorems discovered by W. Schmid \cite{Sch} on the asymptotic behavior of Hodge structure (HS) under the action of the exponential of specific nilpotent transformation twisted with the variable $z$. The nilpotent transformation is assumed to preserve the Hodge filtration. The theorem of Schmid is closely related to the problem of extending a variation of Hodge Structure (VHS) over singularities and asymptotic properties of their period maps. Precisely, orbit theorems in Hodge's theory concern two theorems. The first one is called the Nilpotent orbit Theorem expressing that the limit of a Hodge structure under the 1-parameter orbit of a nilpotent transformation preserving the structure 'exists' and is asymptotic to the period map (see the ref. \cite{Sch} ). The second theorem is called the $SL_2$-orbit Theorem, mainly expressing that any nilpotent orbit is asymptotic to another nilpotent orbit which arises from a representation of $SL_2$, \cite{P}.

Let $S$ be a smooth $L$-variety, and $f:\mathfrak{X} \to S$ a proper smooth morphism. Suppose that this admits a good model over the ring of integers of $L$, i.e. it extends to a proper smooth morphism $f : \mathcal{X} \to \mathcal{S}$ of smooth schemes over the ring of integers $\mathcal{O}_K$. Suppose, moreover, that all the cohomology
sheaves $R^q\pi_*\Omega^p_{\mathcal{X }/\mathcal{S}}$
are sheaves of locally free $\mathcal{O}_{\mathcal{S}}$-modules. The assumption can always be achieved by
possibly enlarging the set $S_0$ of bad primes, [\cite{De}, Theorem 5.5].
The generic fiber of $H^q$
is equipped with the Gauss-Manin connection (by [\cite{NKT}, Theorem 1]) and, again by enlarging $S_0$ if necessary, we may suppose that this extends to a
morphism
\begin{equation} 
\nabla_{GM}:H^q \to H^q \otimes \Omega^1_{\mathcal{S}}.
\end{equation}
Write $\rho_s$ for the Galois representation of $G_L = Gal(\overline{L}/L)$ on the p-adic geometric \'etale
cohomology of $\mathfrak{X}_s$, i.e. $H^*_{et}(\mathfrak{X}_s \times \overline{L}, \mathbb{Q}_p)$.  Fix an archimedean place $\iota : L \hookrightarrow \mathbb{C}$, and fix a finite place $v : L \hookrightarrow L_v$ satisfying:
\begin{enumerate} 
\item if $p$ is the rational prime below $v$, then $p > 2$, and
\item $L_v$ is unramified over $\mathbb{Q}_p$, and
\item no prime above $p$ lies in $S_0$, the set of bad primes.
\end{enumerate}
The Gauss-Manin connection allows us to identify the cohomology of nearby fibers. This applies to the $L_v$ and $\mathbb{C}$ topologies. In fact, both identifications can be described as the evaluation of a single power series with $L$ coefficients, which is convergent both for the $L_v$ and $\mathbb{C}$ topology, cf. \cite{LV}.
On the other hand, under the correspondence of p-adic Hodge theory, the restricted representation $\rho_{s,v}$ corresponds to a filtered
$\phi$-module, namely the de Rham cohomology of $\mathfrak{X}_s$ over $L_v$ equipped with its Hodge filtration and a semilinear Frobenius map. The variation of this filtration is described by a
period mapping; in this setting, this is a $L_v$-analytic mapping
\begin{equation} \label{eq:map}
\Pi: \text{residue disk in} \ S(L_v) \ \longrightarrow \ L_v \ \text{points of a flag variety,}
\end{equation}
Therefore, the variation of the p-adic representation $\rho_{s,v}$ with $s$ is controlled by \eqref{eq:map}. The construction of the map \eqref{eq:map} is based on the existence of an integral model for $\mathfrak{X} \to S$ over the ring of integers of $L$ and the Gauss-Manin connection. One also needs the comparison isomorphism of G. Faltings between the \'etale and crystalline cohomology. It follows that in case $f$ is defined over a p-adic field, fixing a point $0 \in S$, one obtains a well-defined map $$\Pi:\widehat{B'(r)} \to \mathcal{D}, \qquad r \ \text{small}$$ from a universal cover of a local punctured disc around $0$ to a classifying space $\mathcal{D}$. In this text, we do an asymptotic theory for the map $\Pi$ analogous to the complex case. As we will see, most of the constructions known for the complex case can also be equally carried over to the p-adic settings. 
\subsection{Contributions} 
We state and prove three-orbit theorems for the p-adic Hodge structures in a variational setting. The contributions appear in section \ref{sec:main-result} as Theorems \ref{thm:nilpotent-orbit} and \ref{thm:mixed-orbit}. We also explain how a similar theorem can be stated for mixed p-Hodge structure. The first result, namely the p-adic nilpotent orbit theorem [Theorem \ref{thm:nilpotent-orbit}, employs the notion of a 1-parameter orbit defined by the nilpotent transformation $N$ of a $\phi$-isocrystal. Moreover, this studies the asymptotic behavior of the action of the orbit on points in the period domain. 

Consider the family $\mathfrak{X} \to S$ of projective algebraic varieties defined over a p-adic field. We associate a period mapping $\Pi: \widehat{B'(r)} \to \mathcal{D}$, where $\mathcal{D}$ is a classifying space (period domain) of p-adic Hodge structure, and $\widehat{B'(r)}$ is the universal cover of a punctured disc of small redius on $S$- see Section \ref{sec:p-adichodgetheory} below, \cite{Hansen2016PeriodMA, padicoeter}. One can define a notion of GIT stability for the points on $\mathcal{D}$. We shall study an asymptotic behavior of stability on the period domain $\mathcal{D}$. We prove an analog of the W. Schmidt nilpotent orbit theorem for the period mapping $\Pi$ for p-adic Hodge structures. This means that a limit p-adic Hodge structure $F_{\infty}$ can be defined locally attached to a nilpotent orbit. Thus we obtain analogous asymptotic Hodge theory for the p-adic varieties. The result states that the limit of a semistable point $F_0 \in \mathcal{D}$ is also a semistable filtration.  
Denote this limit by $F_{\infty}$. Moreover, the nilpotent orbit of $ F_{\infty}$ is asymptotic to the period map $\Pi$. The nilpotent orbit theorem comes with some essential functional estimates in the Hodge theory. In our case, these estimates appearing in the Proposition \ref{thm:fourier-estimate} are related to p-adic Fourier theory. 

An $SL_2$-orbit is a specific sub-nilpotent orbit extracted from a nilpotent orbit. Specifically, we mean that the extracted orbit is defined via a representation of $SL_2$. The formulation of the $SL_2$-orbit theorem is almost a formal context, following the details from the complex case. The context applies over any field of characteristic zero. More specifically, There exists a representation $\psi_p:SL(2,\mathbb{C}_p) \to G_{\mathbb{C}_p}$, that satisfies,
\begin{equation}
\psi_{p,*}(X_+) \in \mathfrak{g}_{-1}, \qquad \psi_{p,*}(Z) \in \mathfrak{g}_{0}, \qquad \psi_{p,*}(X_-=N) \in \mathfrak{g}_{1},
\end{equation}
where $X_+,Z,X_-=N$ are $\mathfrak{sl}_2$-triples. The $SL_2$-orbit theorem states that this representation is compatible with the nilpotent orbit via certain analytic maps. The conclusion is that an $SL_2$-orbit can be extracted as a sub-orbit of the original orbit, which is asymptotic.

Finally, we stress how all the above concepts must be defined in the mixed setup, i.e., when the underlying variation is of mixed type. The reader should note that in the mixed case, the period map can have essential singularities that show different behavior than the pure case, \cite{P}. In this case, we define $N$-admissible nilpotent orbit, enabling us to state the orbit theorems when the admissibility conditions hold.
\subsection{Tools}
Our method to prove the orbit theorems in p-adic Hodge theory is new and different from the complex case. This text uses three main tools to prove the orbit theorems. First, because the distance function on the period domains of p-adic Hodge structure is a non-archimedean metric, we eliminate the ad-hoc difficulties that appear in the complex case. However, we employ technical aspects from p-adic Hodge theory. Below we briefly summarize the primary tools used to prove the main result. Further discussions appear in the following sections.
\subsubsection{p-adic Fourier correspondence.} 
One of the main tools in formulating and proving our orbit theorem comes from p-adic Fourier theory. p-adic Fourier correspondence provides a theoretic interpretation of the p-adic unit disc. This correspondence behaves compatibly to applying the Sen theorem [see below] to prove our main result. By the work of Y. Amice \cite{A, ST, BSX}, there is an isomorphism 
\begin{equation} \label{eq:Fourier-corresp}
\begin{aligned} 
B(\mathbb{C}_p) \  &\stackrel{\cong}{\longrightarrow} \{\mathbb{Q}_p
 \text{-analytic characters} \ \mathbb{Z}_p \to \mathbb{C}_p \} =: \widehat{\mathbb{Z}_p},\\
 z &\longmapsto \kappa_z(x):=(1+z)^x
 \end{aligned}
\end{equation}
where $\widehat{\mathbb{Z}_p}$ is the character group of the ${\mathbb{Z}_p}$. According to this correspondence, we will regard the points of the $p$-adic unit disc as characters. The isomorphism \eqref{eq:Fourier-corresp} gives a representation interpretation of the unit disc in $p$-adic Hodge theory. When $\mathbb{Q}_p$ is replaced by a finite extension $L/\mathbb{Q}_p$, some difficulties appear in expressing the exact formula for the character $\kappa_z$. There is no explicit formula over a finite extension $L/\mathbb{Q}_p$. When working over a p-adic field $L$, the cyclotomic character is replaced by its analog,
\begin{equation} 
\chi_L: G_L \to \mathcal{O}_L^{\times}, \qquad H_L=\ker(\chi_L)
\end{equation} 
called Lubin-Tate character. Taking a uniformizer $\pi \in \mathcal{O}_L$, the inverse system of residual quotients of $\mathcal{O}_L$ by powers of $\pi$ defines a formal group law, or a $p$-divisible group. Over $L$, there is no explicit formula for the Lubin-Tate character as in \eqref{eq:Fourier-corresp}. However, the whole correspondence in \eqref{eq:Fourier-corresp} remains valid, \cite{ST}. There is an action of $\Gamma=\mathcal{O}_L^{\times}$ on the ring of bounded analytic functions of the punctured disc with the Frobenius action to be $\phi=\pi_*$. In this case, we have the embedding $\mathcal{O}_L(B_L) \hookrightarrow B_{dR}$, where $B_{dR}$ is the Fontaine ring, [see \cite{B, BP, I} for instance]. We have,
\begin{equation} 
\begin{aligned} 
&\mathcal{O}_L \cong \mathbb{Z}_p \times ... \times \mathbb{Z}_p, \quad  \text{as}\ \mathbb{Q}_p \text{-Lie groups} \\
&B_L \hookrightarrow B_{\mathbb{Q}_p} \times ... \times B_{\mathbb{Q}_p},  \quad \text{as analytic subvarieties} \\
&\mathcal{O}_L (B_L) \twoheadleftarrow \mathcal{O}_L(B_{\mathbb{Q}_p} \times ... \times B_{\mathbb{Q}_p}).
\end{aligned}
\end{equation}
The geometry of the disc $B$ can be explained by the analytic characters of $\mathcal{O}_L$. By the above formalism, we identify the points of the $p$-adic disc with analytic characters. As a result, the Fourier correspondence is natural and quite common in p-adic Lie theory. This motivation is compatible with Sen's formalism and the Grothendieck monodromy theorem.  
\subsubsection{Mumford GIT stability} 
We wish to study the semistability of the points in the period domains of the p-adic Hodge structure under the nilpotent orbits. The period domains are simply flag varieties parametrizing the Hodge filtrations of the same type. The associated flag variety has a natural transitive $G$ action, where $G$ is a Lie group defined over p-adic numbers. In this case, specific facts from geometric invariant theory (GIT) can be applied to study the semistability of the points on the orbits in the period domains. GIT systematically assigns a moment map to the $G$-action over a projective variety. Assume the group $G$ acts on the space $\mathcal{D}$ and let,
\begin{equation}
\mu : \mathcal{D} \to \mathfrak{g}=\text{Lie}(G)
\end{equation}
be the moment map, associated with the action of $G$ on $\mathcal{D}$. The properties of the nilpotent orbits in $\mathcal{D}$ can be studied via the map $\mu$. Mumford's semistability criterion is essential in studying the asymptotic theory of 1-parameter orbits. This formalism applies to the nilpotent orbits of the period domains. The limit of a nilpotent orbit can be characterized as the extreme point of the gradient flow of a specific Morse function. Several GIT criteria describe different types of stability properties for the points of the space $\mathcal{D}$. One of these criteria is the well-known Kempf-Ness theorem gives necessary and sufficient conditions for the semistability of a point. It can be briefly written as the following, 
\begin{equation} \label{eq:kempf-ness}
x \in X^{ss} \quad \iff \quad G. x \cap \mu^{-1}(0) \ne \varnothing
\end{equation}
where $\mathcal{D}$ is the locus of semistable points on $\mathcal{D}$, the conditions in \eqref{eq:kempf-ness} are equivalent to that; the orbit $G.x$ to be closed, i.e., it contains its limit points, \cite{VRS, L, ZA}. We use this chain of equivalences to prove the asymptotic stability of a nilpotent orbit in the period domain $\mathcal{D}(\nu)$.  
\subsubsection{Sen formalism in p-adic Hodge theory}
Sen formalism in p-adic Hodge theory concerns a method to associate a specific (unique) operator (called Sen operator, denoted $\Theta$ in below) to a representation of the (absolute) Galois group of a p-adic field $L$. Sen theorem is one of the major tools in the asymptotic theory of p-adic representations. It also allows for relating local systems of p-adic representations to the Gauss-Manin connections. The Sen operator $\Theta$ associated with a representation $\rho$ of the absolute Galois group characterizes the representation $\rho$ up to isomorphism and gives a powerful tool to study the asymptotic properties of the p-adic Hodge structure. The original Sen theorem is more general than the one we employ in this article. We only state a version for Galois representations over $\mathbb{C}_p$. The statement of Sen theorem we use is as follows [see \cite{FO, SE, B} for more detailed discussions]. Suppose that $L$ is an $p$-adic field. If 
\begin{equation}
\rho:G_L \to Aut_{\mathbb{C}_p}(V)
\end{equation}
is a $\mathbb{C}_p$-representation of $G_L$ on the $\mathbb{C}_p$ vector space $V$. Then, there exists a unique operator $\Theta \in End(V)$ where the representation $\rho$ can be expressed locally as:
for every $v \in V$, there exists an open subset $U_{v} \in G_L$, such that,
\begin{equation}
\rho(\sigma)v=\exp[\Theta \log \chi(\sigma)]v, \qquad \sigma \in U_{v}
\end{equation}
where $\chi:G_L \to \mathbb{Z}_p^{\times}$ is a cyclotomic character. The operator $\Theta$ is the Sen operator associated with the representation $\rho$. One can show that two $\mathbb{C}_p$ representations are isomorphic if and only if their corresponding Sen operators are similar. It is a famous result by Sen that the Lie algebra $\mathfrak{g}=Lie(\rho(G_L))$ is the smallest $\mathbb{Q}_p$-space $S \subset End_{\mathbb{Q}_p}V$ such that $\Theta \in S \otimes_{\mathbb{Q}_p} \mathbb{C}_p$. We must stress that the Sen theorem can be stated in a more general setup, and the above version is a simple case that we use in this article.
\subsection{Summary}
Section \ref{sec:preliminaries} consists of preliminaries from p-adic Hodge theory. We first explain the Hodge-Tate decomposition, which is the analog of Hodge decomposition in the complex case. Next, we explain the Weil conjectures in $l$-adic cohomology and the weight filtration defined by Deligne. Then, we proceed with some brief about $(\phi,\Gamma)$-modules and natural correspondences between them and p-adic representations of the Galois groups. We present classical period domains over p-adic fields in section \ref{sec:period-domain} as a setup to express our main results. Next, we provide some basics of the notion of stability over period domains in section \ref{sec:stability}. In Section \ref{sec:fourier} we explain the p-adic Fourier correspondence of Amice and Schneider. Section \ref{sec:main-result} contains our main results and all the material there are our contributions. We state three main theorems, nilpotent orbit, $SL_2$-orbit theorems in the p-adic settings, and nilpotent orbits in the variation of mixed Hodge structure in p-adic settings. Their proofs employ all the materials briefly explained before in the text.
\section{Preliminaries} \label{sec:preliminaries}
This section briefly explains some basic definitions and aspects of p-adic Hodge theory. Then, the discussion summarizes known facts, and the reader is advised to look at the references for detailed analysis. A p-adic Hodge structure is analogous to the complex Hodge decomposition. The Hodge pieces appear as eigenspaces of the Frobenius map acting on certain Galois representations. p-adic Hodge theory studies specific correspondences between Galois representations and $(\phi, \Gamma)$-modules.
\subsection{p-adic Hodge theory} \label{sec:p-adichodgetheory}
\textbf{1. Hodge-Tate decomposition} The $p$-adic \'etale cohomologies of a projective variety $X$ defined over a $p$-adic field $L$ satisfy the decomposition,
\begin{equation}\label{eq:uno}
H_{\textrm{\'et}}^{k}(X_{\overline{\mathbb{Q}_p}},\mathbb{Q}_p) \otimes \mathbb{C}_p= \bigoplus_{i=0}^k  H^{k-i}(X,\Omega_{X/ L}^i) \otimes_{\mathbb{Q}_p} \mathbb{C}_p(-i).
\end{equation}
We denote the $i$-th summand on the right-hand side of \eqref{eq:uno} by $H^{i,k-i}(X/L)$. The decomposition is called \textit{Hodge-Tate decomposition} and results from the convergence of the Hodge to de Rham spectral sequence, 
\begin{equation}
E_1^{p,q}= H^q(X, \Omega_{X/L}^p) \Rightarrow H_{dR}^{p+q}(X/L).
\end{equation}
The groups $H_{\textrm{\'et}}^k(X_{\overline{\mathbb{Q}_p}},\mathbb{Q}_p)$ are $\mathbb{Q}_p$-vector spaces and Galois modules over $G_L=Gal(\overline{\mathbb{Q}_p}/L)$, i.e. we have a continuous representation of the absolute Galois group,
\begin{equation}
\varrho:Gal(\overline{\mathbb{Q}_p}/L) \to Gl(H_{\textrm{\'et}}^k(X_{\overline{\mathbb{Q}_p}},\mathbb{Q}_p))
\end{equation}
called a Galois representation. The twist with $\mathbb{C}(-i)$ means a factor $\chi^i:G_{\mathbb{Q}_p} \to \mathbb{Z}_p^{\times}$ where $\chi$ is the cyclotomic character, acting on the cohomology. The numbers 
\begin{equation} 
h^{r,s}=\dim_{\mathbb{C}_p} H^{i,k-i}(X/L) \otimes_{\mathbb{Q}_p} \mathbb{C}_p(-i)
\end{equation} 
are called Frobenius-Hodge numbers, analogous to the complex Hodge numbers, \cite{Fa1, Fa2, Gr, O1, FM, FO1, FO2, KL, I, Ta, Be, C, Y}. When we have a family of projective varieties defined by a fibration $\mathfrak{X} \to S$ of quasi-projective varieties $X$ and $S$ over $L$, then the decomposition in \eqref{eq:uno} of the $L$-th \'etale cohomology of each fiber defines a filtration,
\begin{equation}\label{eq:filtration_}
F_s^i= \bigoplus_{r \geq i}  H^{k-r}(X,\Omega_{X/ L}^r) \otimes_{\mathbb{Q}_p} \mathbb{C}_p(-r).
\end{equation} 
of the same type. However, the isomorphism \eqref{eq:uno} does not exist in families. This is analogous to the Hodge decomposition over the complex numbers that 
does not vary holomorphically infamilies. One can construct the variety of flags, i.e., a period space denoted $\mathcal{D}(\nu)$ ($\nu=(\nu_i)$ where $\nu_i=\sum_{r \geq i} h^{r,s}$) and define a map, ${S} \to \mathcal{D}(\nu), \  s \mapsto F_s^i,$ parametrizing the filtrations in \eqref{eq:filtration_} \cite{Hansen2016PeriodMA, padicoeter, DOP, R, R1, I, C}. However; the map $\pi$ in general is not well defined, due to the existence of monodromy. 

The analog of the equality
$\dim_{\mathbb{C}} H^i_{dR}(X/\mathbb{C})=  \dim_{\mathbb{Q}_p} H^i_{et}(X,\mathbb{Q}_p)$ 
follows from the comparison isomorphism between singular and de Rham cohomology. In the p-adic case, the situation is slightly more complicated, and the comparison isomorphism only exists after extending scalars to Fontaine's field of p-adic periods $B_{dR}$. If $X$ is only defined over $\mathbb{C}$, it is nontrivial to formulate the correct statement, as there is no natural map 
$\mathbb{C} \to B_{dR}$ along which one can extend scalars. There is however a different way to obtain the desired equality of dimensions. This relies on the Hodge-Tate spectral sequence, \cite{Fa2, padicoeter}.

\vspace{0.3cm}

\textbf{2. Weil conjectures in $l$-adic cohomology} 
Let us suppose that $V_l$ is an $l$-adic representation of the absolute Galois group $G_L$ over $p$-adic field $L$, namely $\varrho:G_L \to GL(V_l)$, where $L$ is the residue field and $I$ is the inertia group. Our example is the \'etale cohomology groups $H^n(X_{\overline{L}}, \overline{Q}_l)$. If $l \ne p$, the semi-stability of $V_l$ means that the inertia subgroup $I(\mathfrak{p})$ at the prime $\mathfrak{p}$ of $\overline{L}$ acts unipotently on $V_l$. Moreover, this action is given by the exponential of a morphism,
\begin{equation}
N:V_l(1) \to V_l
\end{equation}
of $l$-adic representations of $G_L$. It follows that
\begin{equation} 
\varrho(g)=\exp(N. \chi_l(g)), \qquad g \in I(\mathfrak{p}).
\end{equation} 
This is always the case for the $l$-adic \'etale cohomologies, \cite{B, D, KL, Pa, SE, Ta}. When $l=p$, the above formalism is explained via the Sen theorem, \cite{B, Be, Pa, SE}. In this case, the nilpotent transformation $N$ is assumed to satisfy,
\begin{equation} 
NF=q.FN, \qquad q= \text{order of the residue field of L}. 
\end{equation} 
The endomorphism $N$ defines the local monodromy filtration: that is, the unique increasing filtration
\begin{equation} 
... \subset M_i \subset M_{i+1} \subset ...
\end{equation} 
such that:
\begin{itemize} 
\item $N.M_i V(1) \subset M_{i-2} V$,
\item $gr_k^MV(k) \stackrel{\cong}{\longrightarrow} gr_{-k}^MV$. 
\end{itemize} 
By the work of Deligne on Weil conjectures [see \cite{D}], the eigenvalues $\alpha$ of the Frobenius $F$ are algebraic integers. Moreover, there exists integers $w(\alpha)$ such that all complex conjugates of $\alpha$ have absolute value $q^{w(\alpha)/2}$. Let 
\begin{equation} 
W_j= \bigoplus \ \text{generalized eigenspaces of}\ F \ \text{with eigenvalue}\ \alpha \ \text{with}\ w(\alpha)=j. 
\end{equation} 
The filtration 
\begin{equation} 
P_{\bullet}:=\{\ \bigoplus_{j \leq i}\ W_j \ \}
\end{equation} 
is the weight filtration. One can show that the trace of the Frobenius $F_q$ is independent of $l$ in this case, see \cite{D, KL}.
\vspace{0.3cm}

\textbf{3. $(\phi, \Gamma)$-modules in p-adic Hodge theory} \label{sec:phi-N-modules}
The formalism in $p$-adic Hodge theory studies the correspondence between Galois representations and $(\phi, \Gamma)$-modules; here $\phi$ stands for the Frobenius map and $\Gamma$ is a subgroup of the absolute Galois group, \cite{B, FO, Pa}. Let ${E}(B)$ be the norm completion of the ring of bounded analytic functions on the anulus $\ 0< | z | <1$. It has a structure of a filtered $(\phi,\Gamma)$-module where $\Gamma=\mathbb{Z}_p^{\times}$ and $\phi=p_*$ is induced by multiplication map by $p$; \cite{A}, \cite{BSX}, \cite{ST}. We denote the Hodge filtration on $E(B)$ by $F^{\bullet}$. Let 
\begin{equation} 
\chi:G_{\mathbb{Q}_p} \to \mathbb{Z}_p^{\times}
\end{equation} 
be the cyclotomic character and set $H_{\mathbb{Q}_p}:=\ker(\chi)$. By a theorem of Fontaine \cite{FO}, \cite{ST}, there is a one to one correspondence 
\begin{equation} \label{eq:1-1}
\begin{aligned}
\{\mathbb{Z}_p \ \text{representations of} \ G_{\mathbb{Q}_p}\} &\stackrel{\cong}{\longrightarrow} \{ \text{\'etale} \ (\phi,\Gamma)-\text{modules over} \  {E}(B)\}\\
V \qquad &\longmapsto \qquad ({{E}}(B) \otimes_{\mathbb{Z}_p} V)^{H_{\mathbb{Q}_p}}.
\end{aligned}
\end{equation}
where the right-hand side denotes the elements of ${{E}}(B) \otimes_{\mathbb{Z}_p} V$ fixed by $H_{\mathbb{Q}_p}$. Both sides of \eqref{eq:1-1} are equipped with filtrations and a nilpotent transformation $N$. The filtrations and the nilpotent transformations $N$ correspond to each other in \eqref{eq:1-1}. We also assume that $N$ preserves the filtration as follows,
\begin{equation} 
N. F^i \subset F^{i-1}, \qquad \text{transverse property}.
\end{equation} 
In this case, the map $N$ on the right-hand side of \eqref{eq:1-1} is called the $p$-adic monodromy; see \cite{ST, A, Pa, I, C, KL, BP}.

In $p$-adic Hodge theory, we are involved with a series of natural $1-1$ correspondences between $(\phi,\Gamma)$-modules and different $\mathbb{Z}_p$-linear representations satisfying de Rham, crystalline or semistability conditions, \cite{BSX}. There is an analog of this map in the crystalline setting, namely $N:V_p(1) \to V_p$ of $p$-adic representations or the local Galois group $G_K$ called the $p$-adic monodromy morphism. There is no hope of having such a $p$-adic monodromy morphism for every semistable $p$-adic Galois representation. This analogy is explained by the notion of Fontaine or $(\Phi,N)$-modules. A $(\Phi,N)$-module is a $L$-vector space $D$ equipped with,
\begin{itemize}
\item a $\sigma$-linear map (called Frobenius) $\Phi:D \to D$ 
\item a $L$-endomorphism (called monodromy) $N:D \to D$ 
\item $N\Phi =p \Phi N$
\item $D$ has an exhaustive separated decreasing filtration $(Fil^iD_K)_{i \in \mathbb{Z}}$.
\end{itemize}
Denote by $M_K(\Phi,N)$ the category of these modules. Let $V$ be a continuous finite-dimensional representation of $G_K$. Define,
\begin{equation}
D_{st}:Rep_{st}(G_k) \to M_K(\Phi,N), \qquad D_{st}(V):=(B_{st} \otimes_{\mathbb{Q}_p} V)^{G_K}
\end{equation} 
where the superfix means the fixed points, and $B_{st}$ is the Fontaine period ring, the image of this functor is called the admissible filtered $(\Phi,N)$-module. We always have $\dim_{\mathbb{Q}_p}V \geq \dim_K D_{st}(V)$ and $V$ is called semistable if $\dim_{\mathbb{Q}_p}V = \dim_K D_{st}(V)$. If this equality holds for $L$ replaced by a finite extension, it is potentially semistable. We say an object $D \in M_K(\Phi,N)$ has transverse monodromy if,
\begin{equation}
N. Fil^i D_K \subset Fil^{i-1}D_K .
\end{equation}
If $D$ is admissible, then the monodromy $N:D \to D$ induces a $L$-linear map,
\begin{equation}
N_{st}:D(1) \to D
\end{equation}
that depends on the trivialization of the Tate twist. The morphism $N_{st}$ is called the monodromy morphism of $D$. One can show that $N_{st}$ is a morphism of filtered $(\Phi,N)$-modules if and only if $D$ has transverse monodromy. Let us $M_K^{a,tm}(\Phi,N)$ be the category of such modules. It is a tannakian subcategory of $M_K^a(\Phi,N)$. The corresponding representations would be denoted by $Rep_{st,tm}(G_K)$. The map $N_{st}$ corresponds to $N_p:V(1) \to V$ via this correspondence. The correspondence in \eqref{eq:1-1} can also be stated with the rings $\mathcal{R}(B)$ or $\mathcal{E}^{\dagger}(B)$ as well. Under all these isomorphisms, the \'etale modules correspond to each other.
\subsection{Prelimineries on $p$-adic Period Domains}\label{sec:period-domain}
This section briefly explains the basic definitions and terminologies of p-adic period domains used in this text. For a more comprehensive discussion on the concept, we refer to \cite{DOP}. Let $W(k)$ be the ring of Witt vectors of the algebraically closed field $L$ of characteristic $p$, and $L$ be its quotient field. Then, a $\phi$-isocrystal over $L$ is a free module over $W(k)$ equipped with a Frobenius-linear map $\phi$. According to a theorem of J. Dieudonn\'e [see \cite{DOP}], when $L$ is algebraically closed, the category of $\phi$-isocrystals over $L$ is semisimple with simple objects $E_{\lambda}=(L^s, \Phi_{\lambda})$ parametrised by $\lambda \in \mathbb{Q}$. The $\lambda$ is called the slope of $E_{\lambda}$. Thus, it follows that a $\phi$-isocrystal $E$ has a unique decomposition:
\begin{equation}
E = \bigoplus_{\lambda \in \mathbb{Q}} E_{\lambda},
\end{equation}
where $E_{\lambda}$ has slope $\lambda=r/s$. The module $E_{\lambda}$ has a basis of the form $v, \phi .v, \dots , \phi^{s-1} v$ and satisfies $\phi^s.v=p^r.v$. This means that there exists a $W(k)$-lattice $M \subset E_{\lambda}$ such that $\phi^s M = p^r M$. The filtration 
\begin{equation}
F^{\beta}= \bigoplus_{\lambda \leq -\beta} E_{\lambda}
\end{equation}
is called the slope or Newton filtration, \cite{DOP, R, R1, R2}. In this text, we denote the slopes by $\nu$ or $\lambda$. The following is an example of $\phi$-isocrystal of crystalline type.
\begin{example}
Take $L=\mathbb{Q}_p(\sqrt[p]{1})$, and $V=L^2 \cong Le_1 \oplus Le_2$. Let us consider the following data: $\phi(e_1)=e_1, \ \phi(e_2)=p^{k-1}e_2$ and 
\begin{equation} 
F^{k-1}=L(e_1+\pi^ie_2), \qquad i \in \{\ 0,...,p-2\ \},
\end{equation} 
and set $N=0$, and the Galois action:
\begin{equation} 
g.e_1=e_1, \qquad g.e_2={\chi}_L(g)^{-i}. e_2
\end{equation} 
for $g \in Gal(L/\mathbb{Q})$. Where $\chi_L$ is the Teichmuller lift of the cyclotomic character. Then, $V$ is a crystalline representation of $G_L$, i.e. it defines $\phi$-isocrystals with HT-weights $\{0,k-1\}$, \cite{B}. 
\end{example} 
If $L$ is algebraically closed, define the map $\nu : \ E \longmapsto \nu(E)=(\nu_1, \dots , \nu_n)$ from isomorphism classes of $\phi$-isocrystals of rank $n$ into
\begin{equation}
   (\mathbb{Q}^n)_{+} :=  \{ ( \nu_1, \dots ,\nu_n ) \in \mathbb{Q}^n \mid  \  \nu_1 \geq 	\dots \geq \nu_n \},
\end{equation}
where $\nu_{i}$ occurs in $\nu(E)$ with multiplicity the dimension of the isocrystal (isotypic also used) component of type $\nu_i$. The map  $\nu$ is an injective isomorphism called the Newton map. The image of the Newton map can be described as follows. Write $\nu \in (\mathbb{Q}^n)_+  $ as $\nu=(\nu_1^{n_1},...,\nu_r^{n_r}), \ \nu_1 > \nu_2 > ... > \nu_r$. Then 
\begin{equation} \label{eq:integrality}
\nu \in \text{Image} \qquad  \iff  \qquad \nu_i.n_i \in \mathbb{Z}, \ \forall i .
\end{equation} 
By ordering the slopes $\nu_1 \geq \nu_2 \geq ... \geq \nu_n$ one can define the Newton polygon of $E$ as the graph of the function,
\begin{equation} 
i \longmapsto \nu_1 +\nu_2+ ...+\nu_i, \qquad 0 \mapsto 0 .
\end{equation}
\begin{proposition} \cite{DOP} \cite{R}
Let $E$ be $\phi$-isocrystals over a scheme $S$ of characteristic $p$. Then, the Newton vector of $E_s$ goes down under specialization, that is, when $E$ is of constant rank the function $s \mapsto \Vert \nu(E_s) \Vert$ is locally constant on $S$. Furthermore, for any $\nu_0$, the set $\{ s \in S \mid \nu(E_s) \leq \nu_0 \}$ is locally closed in $S$. 
\end{proposition}  
When $L$ is the algebraic closure of $\mathbb{F}_q, \ q=p^h$, the Newton polygon is the Newton polygon of the polynomial $\det(1-\phi^h .t)$. The Newton polygon describes the combinatorics of the slopes of an isocrystal simply. 

Now, suppose $V$ is a $L$-vector space of dimension $n$, and define 
\begin{equation}
\mathcal{D}(\nu):= \{ \ 0 \subset V_1 \subset ... \subset V_r = V \mid \text{rank}(gr_F^i(V))=n_i \ \}.
\end{equation}
The group $G=GL(V)$ acts transitively on $\mathcal{D}(\nu)$. The kernel is the parabolic subgroup $P$ fixing a specific fixed flag of type $\nu$. This gives the identity $\mathcal{D}(\nu)=G/P$. In this sense, Berkovich defines a natural analytification functor,
\begin{equation}
\mathcal{D}(\nu) \longmapsto \mathcal{D}(\nu)^{an},
\end{equation} 
into a smooth compact $L$-manifold, which satisfies all the expected compatibility properties op. cit, \cite{DOP}. Moreover, the group $GL(V)$ acts transitively on $\mathcal{D}(\nu)$. Let $\mathcal{D}(\nu)^{ss}$ be the locus of semistable filtrations (see below for definition). 
$\mathcal{D}(\nu)^{ss}$ is the period domain associated to $(V,\nu)$, and is Zariski open in $\mathcal{D}(\nu)$ for basic reasons. The structure theory of flag variety $\mathcal{D}(\nu)$ (as a symmetric space) can be studied via the root system of the Lie group $G=Gl(V)$ and its Lie algebra. In this case, $\mathcal{D}(\nu)$ has the structure of Schubert variety, i.e., it has a stratification by Schubert cells corresponding to closure of torus orbits, \cite{DOP}. Moreover, there is a Berkovich analitification functor mapping $\mathcal{D}(\nu) \mapsto \mathcal{D}(\nu)^{an}$ the period domain into a p-adic analytic manifold. The following proposition explains the basic properties of this functor.
\begin{proposition} \cite{DOP} \cite{R}
Let $E$ be an isocrystal, and $\nu \in (\mathbb{Q}^n)_+$. Then,
\begin{itemize}
\item The subset $\{ \ x \in \mathcal{D}(\nu)^{an} \mid F_x \in \mathcal{D}(\nu)(\overline{k(x)})^{ss} \ \}$, is open in $\mathcal{D}(\nu)^{an}$, and hence it is the underlying space of an analytic space. (The field $k(x)$ is the residue field of the local ring at the point $x$, and the bar means completion). 
\item The map $\mathcal{D}(\nu) \to \mathbb{R}: \ x \mapsto \mu_{F_x}(E \otimes \overline{k(x)})$ is upper semi-continuous and locally constant.
\item There exists a rigid analytic space $\mathcal{D}(\nu)^{ss,rig}$ and a fully faithful functor $\mathcal{D}(\nu)^{ss,an} \to \mathcal{D}(\nu)^{ss,rig}$.
\end{itemize}
\end{proposition} 
A rigid space is a space that is locally isomorphic to an affinoid space on which the Galois group $G$ acts. An affinoid space is a space isomorphic to a subset of $p$-adic disc defined by the zero set of some ideal of convergent power series. A rigid analytic space is the analog of the complex analytic space over a non-archimedean field. Without mentioning it, we will always assume that $\mathcal{D}(\nu)$ is embedded into its rigid analytification via the Berkovich functor. See \cite{DOP}, \cite{SW} for details.  

When $E$ is defined over a noetherian scheme $S$ of characteristic $p$, then the set of points of $S$ where the Newton vector is constant is locally closed in $S$ and defines a finite decomposition of $S$. The period domains over p-adic fields enjoy exclusive properties satisfy,
\begin{equation} \label{eq:staratification}
\overline{\mathcal{D}_{\nu}} \subset \bigcup_{\nu' \leq \nu}  \mathcal{D}_{\nu'},
\end{equation}
called strong stratification property or Newton stratification property, \cite{DOP, R, R1, R2}. The property \eqref{eq:staratification} distinguishes the $p$-adic period domains from the complex ones. There are many structural open questions about these spaces. For instance, a major open question concerns if the closure of each strata $\mathcal{D}(\nu')$ meets $\mathcal{D}(\nu)$ for $\nu' \leq \nu$, \cite{DOP}. The basic example of period domains is the Siegel moduli space of abelian varieties of genus $g$, described below.
\begin{example} 
Assume we have a family of Abelian varieties $A/S$ over a base scheme of characteristic $p>0$. When $l \ne p$ the family of Tate modules $T_l(A_s)$ for $s \in S$ defines a local system of $\mathbb{Z}_l$-modules on $S$. If $l=p$, the Tate modules $T_l(A_s)$ are replaced by Dieudonne modules $M(A_s)$, which are $\phi$-crystals. In this case, the Dieudonne modules are not constant, as $s$ varies. Let $g$ be a positive integer, and $m \geq 3$ prime to $p$ an auxiliary integer. The Siegel moduli space $M_{g,m}$ of genus $g$ over $\mathbb{F}_p$ classifies the isomorphism classes of triples $(A, \lambda,\eta)$, where $A$ is an Abelian scheme of relative dimension $g$ over $S$, $\lambda:A \to \check{A}$ is a polarization, and $\eta$ is a level $m$ structure. The existence of polarization implies that the Newton vector of the fibers lie in 
\begin{equation} 
{(\mathbb{Q}^{2g})}_+^1=\{(\nu_1,...,\nu_{2g}) \in {(\mathbb{Q}^{2g})}_+ \mid \nu_i +\nu_{2g-i+1}=1, \ 0 \leq \nu_i \leq 1 \}.
\end{equation}  
Let $\Delta(\nu)=\{(i,j) \in \mathbb{Z}^2 \mid  0 \leq i \leq g, \sum_{l=1}^i \nu_{2g-l+1} \leq j <i \}$ and $d(\nu)=\sharp \Delta(\nu)$. Then, it is known that $S_{\nu}$ is equidimensional of dimension $d(\nu)$. There exists a stratification of $S$ with locally closed stratas, where the isomorphism class of the Dieudonne module $M(A_s) \otimes_{\mathbb{Z}_p} \mathbb{Q}_p$ are constant, \cite{R}, \cite{DOP}.
\end{example}
There exists analogous (dual) machinery of isocrystals in the language of $p$-divisible groups. A $p$-divisible group may be defined as an inverse system of finite group schemes $\{G_{p^{hn}} \}_n$ over a formal base scheme $S$. In this case, $h$ is called the height, and $p$ is called the characteristic of the group. In this context, two main points are highlighted:
\begin{itemize}
\item There is an anti-equivalence between the category of $p$-divisible groups over $\text{Spec}(k)$ and the subcategory of $\phi$-crystals over $\text{Spec}(k)$ consisting $(E,\phi)$ such that $p.E \subset F.M$.
\item There exists an anti-equivalence between the category of $p$-divisible groups over $\text{Spec}(k)$ up to isogeny and the full subcategory of $\phi$-isocrystals over $\text{Spec}(k)$ with slope between $0$ and $1$. 
\end{itemize}
From the above points, the classifying space of the $\phi$-isocrystals is interpreted as the deformation space of the $p$-divisible group. As a result, we have the following theorem. 
\begin{theorem}\cite{R} \label{thm:p-divisible}
Assume $S$ is a regular scheme of $char=p >0$, and $X$ is a $p$-divisible group over $S$ with constant Newton vector $\nu$. Then $X$ is isogenous to a $p$-divisible group $Y$ which admits a filtration by closed embeddings $0=Y_0 \subset Y_1 \subset ... \subset Y_r=Y$ satisfying integrality conditions \eqref{eq:integrality}. There also exists natural numbers $r_i \geq 0, \ s >0$ such that $\nu(i)=r_i/s_i$ and 
\begin{equation}
p^{-r_i}Fr^{s_i}:Y_i \to Y_i^{(\sigma^{s_i})}
\end{equation}
and $p^{-r_i}Fr^{s_i}:Y_i/Y_{i-1} \to (Y_i/Y_{i-1})^{(\sigma^{s_i})}$ are isomorphisms.
\end{theorem}
The superfix in Theorem means the fixed element under the action. If $X$ is a $p$-divisible group of height $h$ over $S$ of characteristic $p$ s.t \eqref{eq:integrality} holds, we can associate a lisse $p$-adic sheaf of $W(\mathbb{F}_p)$-modules $\mathcal{V}_X=\ \displaystyle{\lim_{\leftarrow}} \mathcal{V}_{X,n}$ for the \'etale topology on $S$ by
\begin{equation}
\mathcal{V}_{X,n}=\{x \in M/p^nM\ \mid p^{-r}F^s(x)=x\}
\end{equation}
where $M$ is the Dieudonne crystal of $X$. The fibers of $\mathcal{V}_X$ are free $W(\mathbb{F}_p)$-modules of rank $h$. The corresponding $W(\mathbb{F}_p) \otimes_{\mathbb{Z}_p} \mathbb{Q}_p$-adic sheaf depends on the isogeny class of $X$ and corresponds to a representation of the fundamental group,
\begin{equation}
\varrho_X:\pi_1(S) \to Gl_h(W(\mathbb{F}_p) \otimes_{\mathbb{Z}_p} \mathbb{Q}_p).
\end{equation}
The stratification in Theorem \ref{thm:p-divisible} decomposes the representation $\varrho=\oplus_i \ \varrho_i$ by the blocks of size $h_i=\text{height}(Y_i/Y_{i-1})$, cf. \cite{R}.

\begin{remark} \cite{AKT, Sch} \label{rem:analytic-moduli} In this text whenever we talk about the period domain $\mathcal{D}$ we are concerning its underlying analytic space $\mathcal{D}^{an}$ or what is called its analytification. We omit the suffix "an" in most cases. There is a period map $\mathcal{P}: \mathcal{M} \to \mathcal{D}$ where $\mathcal{M}$ is the moduli space of p-divisible groups isogenous to a fixed p-divisible group or by what we said the moduli of $\phi$-isocrystals. We will not deal with this period map, which is associated with a moduli problem on p-divisible groups or similar for $\phi$-isocrystals. In p-adic Hodge theory, there are several notions of period maps. We will come back to this notion again in the next sections. In section \ref{sec:period-map} we introduce a local period map $\Pi:U \to \mathcal{D}$, defined on a local disc $U/L$ which is well defined on a universal cover of $U$. This period map factors trivially through $\Pi$ in the above. Our purpose in this section was to make the definition of period domains in the p-adic setting more precise and set our notations.
\end{remark}

\subsection{Mumford GIT stability}\label{sec:stability}
We propose to study the semistability conditions for the points on the period domains. Below, we define this concept and explain its relation to geometric invariant theory. Finally, we will use a GIT criterion of semistability due to D. Mumford to obtain an asymptotic property of nilpotent orbits. 

If $(V,F)$ is a filtered $L$-vector space. Then, we define the degree and the slope of the filtration $F$ by,
\begin{equation} 
\deg(V,F)=\sum_i i .\dim gr_F^i(V), \qquad \mu(V,F)=\deg(V,F)/rank(V,F),
\end{equation} 
respectively. The degree is an additive function on the category of filtered vector spaces, and also one has 
\begin{itemize} 
\item $\deg(V)=\deg(\bigwedge^{max}V)$,
\item $\mu(V \otimes W)=\mu(V)+\mu(W)$,
\item $\mu(V^*)=-\mu(V)$. 
\end{itemize} 
We employ the following notion of semistability for a filtration of an $\phi$-isocrystal.
\begin{definition} 
A filtration $F^{\bullet}$ of an isocrystal $M$ is called semistable if $\mu(N,F^{\bullet}) \leq \mu(M,F^{\bullet})$ for any subisocrystal $N$ of $M$.
\end{definition}
Assume that the filtration $F$ of $V$ is split semisimple, that is $V=\bigoplus_i  V_i \ $ with $V_i$ of pure type. This decomposition corresponds to a $1$-parameter subgroup (1-PS) of $G=GL(V)$ namely,
\begin{equation} 
\lambda:\mathbb{G}_m \to SL(V).
\end{equation} 
By considering a suitable basis $\{s_0,...,s_n\}$ of $V$ we can write $\lambda $ diagonally as 
\begin{equation}
\lambda(t)=diag[t^{\rho_0},...,t^{\rho_n}]t^{-\rho}, \qquad \rho_0 \geq ...\geq \rho_n, \ \rho=\sum \rho_i/(n+1). 
\end{equation} 
Mumford introduced a subsheaf 
\begin{equation} 
{L}(\nu) =(t^{\rho_0}s_0,...,t^{\rho_n}s_n) \subset \mathcal{O}_{X \times \mathbb{A}^1}(1)
\end{equation} 
generated by the sections $t^{\rho_i}s_i$. D. Mumford gives an equivalent formulation of semistability by defining a GIT slope $\mu^{L(\nu)}(F,\lambda)$ for a point $F \in \mathcal{D}(\nu)$ and the 1-PS $\lambda$, where $L(\nu)$ is a line bundle on $\mathcal{D}(\nu)$ associated to $\nu$; see \cite{DOP}. The following theorem expresses a GIT criterion for semistability.
\begin{theorem} \cite{DOP}
The point $F \in \mathcal{D}(\nu)$ is semistable if and only if for all 1-PS $\lambda$ of $GL(n)$, 
\begin{equation}
\mu^{L(\nu)}(F,\lambda) \geq 0.
\end{equation}
\end{theorem} 
Let us denote $X=G/P$ and $x=F$, where $G=GL(n)$. Also, let,
\begin{equation}
\mu : X \to \mathfrak{g}=\text{Lie}(G)
\end{equation}
be the moment map, associated with the action of $G$ on $\mathcal{D}(\nu)$, \cite{VRS, L, ZA}. The properties of the nilpotent orbits in $X$ can be studied via the map $\mu$. One defines a function,
\begin{equation}
W_{\mu}(x,\zeta)=\langle \mu(\exp(\zeta. \chi_z).x),\zeta \rangle, \qquad x \in X , \ \zeta \in \mathfrak{g} \backslash {0}
\end{equation}
called the $\mu$-weight of the pair $(x, \zeta)$, where $\langle . , . \rangle$ is an inner product in $\mathfrak{g}$. The semistablity of $F$ can be studied via the critical points of $W_{\mu}(x,\zeta)$. Moreover, it is related to orbit closures by the Kempf-Ness theorem presented next.
\begin{theorem} (Kempf-Ness Theorem) \cite{VRS} \label{th:Kempf} 
The point $x$ is semistable if and only if one of the following equivalent conditions holds:
\begin{equation}
\text{the orbit}\ G.x \ \text{is closed} \ \ \Leftrightarrow \ \ G. x \cap \mu^{-1}(0) \ne \varnothing   .
\end{equation} 
\end{theorem}
Moreover, it is easy to show that if $F \in \mathcal{D}(\nu)$ and $F_0 \in G.x_0$ then there exists an $N$ such that $\exp(N. \chi_z).x_0=F$. Therefore, in Theorem \ref{th:Kempf}, one needs to consider the one parameter orbits in the first left criteria. The reader may find various other formulations and numerical invariants in the reference; see also \cite{L, ZA}. 
\subsection{$p$-adic Fourier correspondence}\label{sec:fourier}
The material in this section are quite well known in p-adic Hodge theory. The references are \cite{A}, \cite{BSX}, \cite{ST}. Let $L$ be a $p$-adic field with a ring of integers $\mathcal{O}_L$. Y. Amice \cite{A} gives a representation-theoretic interpretation of the p-adic unit disc. He shows that the geometry of the $p$-adic unit disc $B_L(\mathbb{C}_p)$ can be explained by the analytic characters of $\mathcal{O}_L$ via Fourier analysis. 
\begin{theorem}\cite{ST} \label{thm:amice}
There is an isomorphism 
\begin{equation} \label{eq:fourier-corres}
B(\mathbb{C}_p) \  \stackrel{\cong}{\longrightarrow} \{ \mathbb{Q}_p \ \text{analytic characters} \ \mathcal{O}_{\mathbb{Q}_p} \to \mathbb{C}_p \} =: \widehat{\mathcal{O}_{\mathbb{Q}_p}}
\end{equation}
\begin{center}
$z \longmapsto \kappa_z(x):=(1+z)^x$
\end{center}
where, $\widehat{\mathcal{O}_{\mathbb{Q}_p}}=Hom_{an}(\mathcal{O}_{\mathbb{Q}_p}, \mathbb{C}_p^{\times})$.
\end{theorem}
We have stated the Fourier isomorphism over $\mathbb{Q}_p$ as stated by Amice \cite{A}, where the converse map is given by $\chi \mapsto \chi(1)$,  \footnote{A similar isomorphism can be stated over a $p$-adic field $L$. However, the specific formula stated in (43) does not exist over finite extensions of ${\mathbb{Q}_p}$, \cite{ST, BSX},}. We can write $\kappa(z)=\kappa_z$ by a formal power series, 
\begin{equation}
\kappa(z)=1+\alpha.z+... \in z. \mathcal{O}_{\mathbb{Q}_p}[[z]]
\end{equation}

The element $\alpha \in \mathcal{O}_{\mathbb{Q}_p}$ depends on parametrization. Different choices of parameters will correspond to multiplication by an element in $\mathcal{O}_{\mathbb{Q}_p}$ on $\alpha$. \footnote{One should not confuse the two notations $\mathcal{O}_B$ and $\mathcal{O}_L$. There exists a topological duality between them.} 

The Amice isomorphism in (43) is related to the p-adic Fourier transform, which we explain over an arbitrary p-adic field $L$. There is a Fourier transform,
\begin{equation}
\mathcal{D}:D(\mathcal{O}_L,\mathbb{C}_p) \stackrel{\cong}{\longrightarrow} {\mathcal{O}_B}, \qquad \lambda \mapsto F_{\lambda}
\end{equation}
where $D(\mathcal{O}_L,\mathbb{C}_p)=C^{an}(\mathcal{O}_L,\mathbb{C}_p)'$ (topological dual), $\mathcal{O}_B$ is the ring of analytic functions on the disc $B$ and $F_{\lambda}$ is defined via the correspondence (19) by
\begin{equation}
F_{\lambda} : \widehat{\mathcal{O}_L}(\mathbb{C}_p) \to \mathbb{C}_p, \qquad \chi \mapsto \lambda(\chi).
\end{equation}
The Fourier transform also provides a pairing
\begin{equation}
\{\ ,\ \}:\mathcal{O}(B/\mathbb{C}_p) \times C^{an}(\mathcal{O}_L) \to \mathbb{C}_p.
\end{equation}
The space of analytic functions $f:\mathcal{O}_L \to \mathbb{C}_p$ has the structure of a Hilbert space with respect to the above pairing. The following formulas hold for $\{\ ,\ \}$:
\begin{itemize}
\item $\{1 ,f\}=f(0)$.
\item $\{F,\kappa_z \}=F(z)$.
\item $\{F ,\kappa_z .f \}=\{F(z+\ ) ,f \}$.
\item $\{F ,f(a.\ ) \}=\{F \circ [a] ,f \}$. 
\item $\{F ,f' \}=\{\Omega \log F  ,f \}$ 
\item $\{F ,xf(x ) \}=\{\Omega \partial F ,f \}$
\item $\{F ,P_m(\ \Omega) \}=(1/m!)d^mf/dz^m(0)$
\end{itemize}
\begin{theorem}\cite{ST}
Any analytic function $f:\mathcal{O}_L \to \mathbb{C}_p$ has a unique representation in the form
\begin{equation}
f=\sum_{n=0}^{\infty} c_n P_n(\ . \alpha)
\end{equation}
where $c_n=\{f , z^n\}$. Furthermore, such a series is convergent provided that there exists a real number $r$ such that $\mid c_n \mid r^n \to 0$ as $n \to \infty$. 
\end{theorem}
Define the norm $\parallel . \parallel$ on the space of power series by $ \parallel f\parallel_{a,n}=\displaystyle{\max_{z \in  a+\pi^n \mathcal{O}_L}} \mid f(z) \mid $. 
\begin{proposition} \cite{ST} \label{thm:fourier-estimate}
Consider the Fourier expansion of the function 
\begin{center}
$\ \exp(. \ \alpha \log z)=\sum P_m(. \ \alpha)z^m$ 
\end{center}
There are constants $C_1$ and $L$ such that 
\begin{equation} \label{eq:fourier-bound}
\parallel P_l(y. \alpha) \parallel_{0,n} <C_1 p^{-k.l/q(n)} , \qquad n \geq 1
\end{equation}
where $q(n) \to \infty$ as $n \to \infty$. 
\end{proposition}

The Fourier expansion for $H(x,y)=\exp(y \log x)$ has a simple form that looks like the Taylor expansion in analysis. The function $H(x,y)=\exp(y \log x)$ is a rigid analytic function on $B(r)^2$ for $r$ sufficiently small. In fact the Fourier expansion for $H(x,y)$ has a simple form that looks like the Taylor expansion in analysis. We have the power series expansion $\exp(y \log(z))=\displaystyle{\sum_{n=0}^{\infty}P_n(y)z^n}$, where the polynomials $p_n(y)$ satisfy the following properties,
\begin{itemize}
\item $p_0(y)=1, \ p_1(y)=y$.
\item $p_n(0)=1, \ n \geq 1$.
\item $\deg(p_n)=n$ and the leading coefficient of $p_n$ is $1/n!$.
\item $p_n(y+y')= \sum_{i+j=n}p_i(y)p_j(y')$
\item $p_n(\partial).f(x)|_{x=0}=(1/n!)(d^nf/dx^n)|_{x=0}$, where $f(x) \in \mathbb{C}_p[[x]]$
\end{itemize}
The last property can be obtained by a comparison to the formal Taylor series; let $\delta=d/dx$ then the Taylor formula reads as  $\exp(\delta b)h(a)=\sum_n \dfrac{\delta^n}{n!}h(a)b^n=h(a+b)$. Inserting $a=\log(x), \ b=\log(y), \ h=f \circ \exp$ one gets $\exp(\partial\log(y))f(x)=f(x+y)$. Therefore, the orbit limit we are going to compute with is an analogue of formal Taylor series expansion, \cite{ST, BSX}.

We may assume $r$ the radius of the disc $B(r)$ is small enough such that $\log:B(r) \to B(r)$ is an isomorphism, and we have
\begin{center} \label{diag:log}
$\begin{CD}
B(r) @>{\cong}>> B(r) \\
@V{z \to \kappa_z}VV @VV{z \mapsto \Omega .z}V \\
\widehat{\mathcal{O}}_L(r \Omega) @>>{\kappa_z \to \log \kappa_z(1)}> B(r \mid \Omega \mid )
\end{CD}$
\end{center}
The $\log$ is the logarithm of the associated formal group law. To explain this diagram, first note that the map $z \mapsto \kappa_z$ maps $B(r) \to \widehat{\mathcal{O}_L}(r \mid \Omega \mid )$ and is a rigid analytic isomorphism. Thus, the diagram says under the isomorphism $B(\mathbb{C}_p) \stackrel{\cong}{\rightarrow} \widehat{\mathcal{O}}_L(\mathbb{C}_p)$ the logarithm functions are compatible, see \cite{ST} and \cite{BSX} for details.

We shall use the p-adic Fourier correspondence to obtain a  representation-theoretic interpretation of the points of the p-adic disc. Together with the Sen theorem mentioned before, this formulation enables us to obtain a new presentation of the twisted period map associated with variation p-adic Hodge structure. Then the inequality \eqref{eq:fourier-bound} will provide estimates appearing in the nilpotent orbit theorem. 

\subsection{Period map} \label{sec:period-map}
This section assumes $L/ \mathbb{Q}$ is a number field. Let $S$ be a smooth $L$-variety, and $f:\mathfrak{X} \to S$ a proper smooth morphism. Suppose that this admits a good model over the ring of integers of $L$, i.e. it extends to a proper smooth morphism $f : \mathcal{X} \to \mathcal{S}$ of smooth schemes over the ring of integers $\mathcal{O}_K$. 
We denote by $\rho_s$ the representation of the
Galois group $G_L$ on the \'etale cohomology group of $\mathfrak{X}_s$:
\begin{equation} \label{eq:rep-etale} 
\rho_s: G_L \to Aut \left (H^q_{et}(X_s \times_L \overline{L}, \mathbb{Q}_p) \right ).
\end{equation}
Consider the assumptions and notations mentioned in the introduction before section \ref{sec:contributions}.

Fix $s_0 \in  \mathcal{S}$. 
Let $V = H^q_{dR}(\mathfrak{X}_{s_0}/L)$. We denote by $V_v$ and $V_{\mathbb{C}}$ the $L_v$- and $\mathbb{C}$-vector spaces obtained
by $\otimes_L L_v$ or $\otimes_{(L,\iota)}\mathbb{C}$. Then $V_{\mathbb{C}}=H^q_{dR}(\mathfrak{X}_{s_0},\mathbb{C})$, which is also (by the comparison theorem) identified with the singular
cohomology $H^q_{sing}(\mathfrak{X}_{s_0},\mathbb{C})$. In particular, monodromy defines a representation $$\varrho : \pi_1(S_{\mathbb{C}}, s_0) \to GL(V_{\mathbb{C}}), \qquad \Gamma:=  \overline{\text{Image}(\varrho)}$$
where the overline means Zariski closure. Then $\Gamma$
is an algebraic subgroup of $GL(V_{\mathbb{C}})$ called the monodromy group at $s_0$. Both $V_{\mathbb{C}}$ and $\Gamma$ depend on the choice of
archimedean place $\iota$.

The Gauss-manin connection provides the following two identifications 
\begin{equation} 
\begin{aligned} 
GM_v &: H^q_{dR}(\mathcal{X}_{s_0} /L_v) \cong H^q_{dR}(\mathcal{X}_s/L_v),\\ 
GM_{\iota} &: H^q_{dR}(\mathfrak{X}_{s_0},\mathbb{C})\cong H^q_{dR}(\mathfrak{X}_s,\mathbb{C})
\end{aligned}
\end{equation}
when $s$ is sufficiently close to $s_0$. In the coordinates of a specific basis, $GM_v$ is given by a matrix with entries
$\mathcal{O}_{L_v}[[z_1, \dots , z_m]],$ convergent in the regions noted above.
The fiber over the $\mathcal{O}_{L_v}$-point $s_0$ of $\mathcal{S}$ gives a smooth proper model $\mathcal{X}_{s_0}$ for $\mathfrak{X}_0$. We have identifications 
\begin{equation} \label{eq:p-adic-isom}
H^q_{dR}(\mathcal{X}_s/L_v) 
 \cong H^q_{dR}(\mathcal{X}_{s_0}/L_v) \cong H^q_{cris}(\mathcal{X}_{s_0}) \otimes_{\mathcal{O}_{L_v}} K_v, \qquad s \in U
\end{equation} 
where, $H^q_{cris}$ is the crystalline cohomology of $\mathcal{X}_{s_0}$, see \cite{Ber, BO}. 
The crystalline cohomology is equipped with a Frobenius operator $\phi_v$, which is semilinear with respect to the Frobenius on the unramified extension $L_v/\mathbb{Q}_p$. By
the isomorphisms \eqref{eq:p-adic-isom}, this $\phi_v$ acts on $H^q_{dR}(\mathcal{X}_y/L_v)$ and $H^q_{dR}(\mathcal{X}_{s_0} /L_v)$ as well, in a
manner compatible with the map $GM_v$.

Assume now $V = H^q_{dR}(\mathfrak{X}_{s_0}/L)$ is equipped
with a Hodge filtration: $$F_0^{\bullet}: \  F^0 \varsupsetneq F^1 \varsupsetneq \dots \ .$$ Let $\mathcal{D}$ be the $L$-variety parameterizing flags in $V$ with the same dimensional data as $F^{\bullet}$,
and let $F_0 \in \mathcal{D}(K)$ be the point corresponding to the Hodge filtration on $V$ we initiated in the above.
With base changing by means of $v$ and $\iota$, we get a $L_v$-variety $\mathcal{D}_v$ and a $\mathbb{C}$-variety $\mathcal{D}_C$. We
denote by $F_{\iota,0} \in \mathcal{D}_{\mathbb{C}}$ the image of $F_0$.
Let $U_{\mathbb{C}}$ be a contractible analytic neighbourhood of $s_0 \in S_{\mathbb{C}}^{an}$. The Gauss-Manin
connection defines an isomorphism $H_{dR}(\mathfrak{X}_s/\mathbb{C}) \cong  H_{dR}(\mathfrak{X}_{s_0}/C)$ for each $s \in U_\mathbb{C}$. In
particular, the Hodge structure on the cohomology of $\mathfrak{X}_s $ defines a point of $\mathcal{D}_{\mathbb{C}}$; this
gives rise to the complex period map

$$\Pi_{\mathbb{C}} : U_{\mathbb{C}} \to \mathcal{D}_{\mathbb{C}}, \qquad 
U_{\mathbb{C}} = \{z: \|z_i\|_{\mathbb{C}} < \epsilon \}.$$

Indeed, $\Pi_{\mathbb{C}}$ extends to a map from the universal cover of $S^{an}_{\mathbb{C}}$ to $\mathcal{D}_{\mathbb{C}}$ and this map is
equivariant for the monodromy action of $\pi_1(S^{an}_{\mathbb{C}} , s_0)$ on $\mathcal{D}_{\mathbb{C}}$. 
Then we have the containment 
$$\Gamma. F_{\iota,0} \subset  \text{the Zariski closure of} \ \Pi_{\mathbb{C}}(U_{\mathbb{C}})$$ 
inside $\mathcal{D}_{\mathbb{C}}$.
For a $v$-adic analog, take $s \in U_v$, then again Gauss-Manin connection allows one to identify the Hodge filtration on $H^q_{dR}(\mathcal{X}_s/L_v)$ with a
filtration on $H^q_{dR}(\mathcal{X}_s/L_v)$, and thus with a point of $\mathcal{D}_{L_v}$. This gives rise to a $L_v$-analytic function
$$\Pi_v : U_v \to \mathcal{D}_{L_v}, \qquad U_v = \{z : \|z_i\|_v < \epsilon \}.$$
In this case also, the image of the orbit of $F_{v,0}$ under the action of (the same) $\Gamma$ is contained in the Zariski closure of the image of the period map $\Pi_v$, cf. \cite{LV}
$$\Gamma. F_{v,0} \subset  \text{the Zariski closure of} \ \Pi_{v}(U_v).$$

To relate the above construction to the representation $\rho_s$ in \eqref{eq:rep-etale}, One uses p-adic Hodge theory, \cite{BC, Fon}.
For each $s \in  U_v$, the representation $\rho_{s}$ is crystalline upon restriction to $L_v$,
because of the existence of the model $\mathcal{X}_s$ for $X_s$. By p-adic Hodge theory, there is [8,
Proposition 9.1.9] a fully faithful embedding of categories

$$\text{crystalline representations of}\ G_{L_v}\  \hookrightarrow \ \text{filtered}\ \phi{-}\text{isocrystals}$$
On the other hand, by the crystalline comparison theorem of Faltings \cite{Fal}, the aforementioned embedding carries $$\rho_{s,v} \longmapsto E_s=\left (H^q_{dR}({X}_s/L_v), \phi_v, \text{Hodge filtration for} \ {X}_s \right ).$$ Combining the comparison isomorphism by the above embedding we deduce that the representation $\rho_{s,v}$ gives rise to a filtered $\phi$-isocrystal $E_s$. 

The conclusion is that in case the fibration $\mathfrak{X} \to S$ is defined over a local field, one can consider the period map $\Pi$ to be defined on local discs $U=B(r)$ in the variety $S$. For the possibility of the existence of degenerations, we generally may assume that the disc is punctured at $0$. Then the period map $\Pi$ is well-defined on a universal cover $\widehat{B'(r)}$ of the punctured disc.
%
\section{Main results} \label{sec:main-result}

One of the significant concepts of Hodge theory is the degeneration of a Hodge structure in a family of projective algebraic varieties $\mathfrak{X} \to S$ over a pointed curve $(S,0)$ defined over a p-adic field $L$. Denote $S^*=S \setminus \{0\}$. Here we assume $X$ and $S$ are quasi-projective varieties. Moreover, we may assume the generic fiber of the family is smooth, and the exceptional fiber of the family is a normal crossing. We are interested in understanding the limits of Hodge structure on the \'etale (or crystalline) cohomologies of the generic fibers $\mathfrak{X}_s$. A classical idea to study the degeneration of the Hodge structure is to consider the nilpotent orbits. One can present this by the nilpotent orbits of the points in period domains of HS of the type appearing for \'etale or crystalline cohomology of a generic fiber $H_{et}^k(X_s)$. We first define an analog of nilpotent orbit for the p-adic HS. Then, we prove a limit semi-stability property for the nilpotent orbits in p-adic case. The result is an analog of that of W. Schmid's nilpotent orbit theorems \cite{KP, P, GS} in complex Hodge structures. 

Consider a family of projective algebraic varieties $\mathfrak{X} \to S$ over a pointed curve $(S,0)$ defined over a p-adic field $L$. Let $\Pi: \tilde{S^*} \to \mathcal{D}$ be the period map associated to this fibration as constructed in \ref{sec:period-map}. We can extend this map by base change over $\mathbb{C}_p$. Denote by $\Pi= \Pi_{\mathbb{C}_p}$ the same map when the coefficients are extended. The domain $\mathcal{D}$ can be regarded as the classifying space of Hodge-type filtrations $F^{\bullet}$ of the same type $(\nu)$ on the filtered isocrystal $E=(V, \phi, F)$ associeted to the de Rham or \'etale cohomology of a fixed generic fiber, in the sense explained in \ref{sec:period-map}. We can assume $E$ is equipped to a nilpotent transformation $N$, see sections \ref{sec:period-domain} and \ref{sec:phi-N-modules}.

\begin{definition} (Nilpotent orbit) \label{def:nilp-orbit}
Let $F_0 \in \mathcal{D}$ be an arbitray point. The nilpotent orbit of $F_0$ associated to $N$, is an analytic map $\eta: \tilde{S} \to \mathcal{D}$ given by,
\begin{equation}
\eta(s)=\exp[N\log(s)]F_0, \qquad F_0 \in \mathcal{D}.
\end{equation}
such that the following conditions are satisfied:
\begin{itemize}
\item $\eta(s) \in \mathcal{D}^{ss}$ for $s \to \infty $.
\item $N \circ \phi=p^r. \phi \circ N$, for some $r$.
\end{itemize}
\end{definition}

Making the limit in the above definition makes sense because the ground variety, $\mathcal{D}(\nu)$, is an analytic space and carries a natural (non-archimedean) metric. In the following theorem we have used the notations of section \ref{sec:fourier}.
\begin{theorem}[p-adic nilpotent orbit theorem] \label{thm:nilpotent-orbit}
In the above settings, consider the local disc $B(r)$
at $0$ on $S$. Consider the local period map $\Pi$ on $B'(r)=B(r) \setminus \{0\}$. Define, 

\begin{equation}\label{eq:th_eq}
\eta(\kappa_{z})= \exp [N \log(\kappa_z)] \Pi( \kappa_{z}).
\end{equation}
to be Then, the followings are true:
\begin{itemize}
\item[(1)] The limit $F_{\infty}:=\displaystyle{\lim_{\kappa_{z} \to 0} \eta(\kappa_z)}$ exists and is a semistable filtration of $V$. 
\item[(2)] $\xi(\kappa_{z})= \exp [N \log(\kappa_z)] F_{\infty}$ is a nilpotent orbit which is asymptotic to the period map $\Pi$.
\item[(3)] For each non-archimedean metric\footnote{A non-archimedean metric is a metric that satisfies the following property known as the triangle axiom $d(\xi, \theta) \leq max [d(\xi, \zeta),d(\zeta,\theta)]$ for any 3 points $\xi,\ \zeta, \theta$ in the space. } $d$ on $\mathcal{D}^{an}$, there exists constants $C$, $L$, $l$ and $r >0$ such that 
\begin{equation}
d(\xi(\kappa_z), \eta( \kappa_z)) <C. p^{-k.l/q(n)}r^n, \quad \forall n ,
\end{equation}
where $q(n) \to \infty$ as $n \to \infty$.
\end{itemize}
\end{theorem}
\begin{proof} 
The group $G=GL(n, \mathbb{C}_p)$ acts transitively on $\mathcal{D}$. We fix a point $F_0 \in \mathcal{D}$. Any other point (for instance $\Pi(\kappa_z)$) in $\mathcal{D}$ can be written as $\varrho(g)F_0$ for some representation $\varrho:G_L \to GL(V_{\mathbb{C}_p})$. By Sen theorem \cite{SE}, we write 
\begin{equation}\label{eq:period-param}
\Pi(\kappa_z)=\exp[\mathfrak{n}( \chi (z))\log \chi (z)]F_0
\end{equation}
locally on $ \mathcal{D}$, where $\chi(z)$ is the cyclotomic character appearing in the Sen theorem, and its variable depends on $z$ continuously. Then, the function $\mathfrak{n}( \chi(z))$ is also locally analytic. We plug \eqref{eq:period-param} into the orbit function \eqref{eq:th_eq},
\begin{equation} \label{eq:twisted-period}
\eta(\kappa_z)= \exp [N \log(\kappa_z)] \exp[\mathfrak{n}( \chi (z))\log \chi(z)]F_0.
\end{equation}
By Proposition \ref{thm:fourier-estimate}, it can be written as the product of two series of the form $\ \sum_m {P}_m z^m \ $ with coefficients to be polynomial functions in finite matrices. Therefore, if the variables in \eqref{eq:twisted-period} are chosen so that $||\kappa_z||$ and $||\mathfrak{n}(\chi_z)||$ are small enough (see \cite{ST}). Then, the uniform bound in \eqref{eq:fourier-bound} will hold for both of the 1-parameter orbits on the right-hand side of \eqref{eq:twisted-period}. This shows that $\eta(\kappa_z)$ is convergent. If $d$ is a non-archimedean metric on the flag variety, then for a suitable constant $C$,
\begin{equation} \label{eq:distance-bound}
d(\xi(\kappa_z), \eta( \kappa_z)) 
\leq C. max (\|\exp [N \log(\kappa_z)]\|, \|\exp [N \log(\kappa_z)] \exp[\mathfrak{n}( \chi(z))\log \chi(z)]\|) 
\end{equation}
Now, the estimates for Fourier coefficients in the Proposition \ref{thm:fourier-estimate} will hold for \eqref{eq:distance-bound} by the non-archimedean property, which implies the inequality in item (3) and also the existence part in item (1). The Kempf-Ness Theorem states that the semistability of $F_{\infty}$ is equivalent to the statement that: the orbits $\lambda(\xi_z).F_{\infty}$ all are closed. That means they contain their limit points. It will suffice to show that the nilpotent orbits of the points $F_{\infty}$ belong to $G.F_{\infty}$. For a general $g \in GL(V)$ let,
\begin{equation}
\hslash(g):=\lim_{\kappa_z \to 0} \exp[-N'\log \kappa_z].g.\exp[N'\log \kappa_z] \ \in G,
\end{equation}
with $N'$ nilpotent. This limit exists, cf. \cite{VRS}. Now if 
$\lim_{\kappa_z \to 0} \exp [N'\log \kappa_z] .F_0=F_{\infty}^{'}$, then choose $g \in G$ such that, $g.F_{\infty}=F_{\infty}'$. We have, 
\begin{equation}
\begin{aligned}
\hslash(g) F_0 &=\lim_{\kappa_z \to 0} \exp[-N'\log \kappa_z].g.\exp[N'\log \kappa_z]. F_0 \\
&=\lim_{\kappa_z \to 0} \exp[-N'\log \kappa_z].g.F_{\infty}'=\lim_{\kappa_z \to 0} \exp[-N'\log \kappa_z].F_{\infty} .
\end{aligned}
\end{equation}
We can interpret the last equation so that the limits $\lim_{\kappa_z \to 0} \exp[-N'\log \kappa_z].F_{\infty}$ has the form $G.F_{\infty}$. That is what is expected to prove. The semistability of $F_{\infty}$ follows from the Kempf-Ness theorem \ref{th:Kempf}.
\end{proof}
We call the filtration $F_{\infty}$ the limit slope filtration in analogy to the complex case of limit Hodge filtration. The convergence of the one-parameter orbit of the period map in \eqref{eq:th_eq} over a homogeneous manifold is a general fact that is formally based on the GIT over homogeneous manifolds. These orbits can be explained as gradient flows of certain Morse functions defined via the inner product or the metric of $V$, [see \cite{VRS}]. The statements of the Theorem \ref{thm:nilpotent-orbit} are the non-archimedean analog of the Schmid nilpotent orbit theorem \cite{S, KP, P}. From the proof of Theorem '\ref{thm:nilpotent-orbit}, it can be understood that a similar statement holds for nilpotent orbits in several variables,
\begin{equation}
\eta(\kappa_{z_1},...,\kappa_{z_r}):=\exp [ N_1 \log(\kappa_{z_1})+...+N_r\log(\kappa_{z_r})]\Theta(\kappa_{z_1},...,\kappa_{z_r}),
\end{equation}
where $N_i$ are commuting nilpotents; they also exist and satisfy similar estimates. The different limits corresponding to the choice of elements in the nilpotent cone define boundary points in the period space $\mathcal{D}(\nu)^{ss}$. 
\begin{remark}
In Theorem \ref{thm:nilpotent-orbit} and its proof, the limit can be understood in the $p$-adic unit disc or its covering by the exponential map. On the disc, the limit goes $\kappa_z \to 0$; however, going upstairs in the \ref{diag:log}, we have $\log \kappa_z \to \infty$ in a formal way. The meaning of the limit has to be understood via the Amice-Schneider-Teitelbaum canonical (Fourier) isomorphism encountered in Theorem \ref{thm:amice}. That is $\kappa_z \to 0$ means $z \to 0$ in the preimage (left side) of \eqref{eq:fourier-corres}. We have formulated the asymptotic by the limits of cyclotomic characters instead of usual numbers because of the theorem's compatibility with other facts in the $p$-adic Hodge theory. The correspondence in Theorem \ref{thm:amice} is formally written in the unit disc around $0$, but it corresponds to analytic characters that are functions around the point $1$. This phenomenon is formal, and we can even subtract the values of $\kappa_z$ by $1$. In $p$-adic Fourier theory, the characters on the right side of \eqref{eq:fourier-corres} will be considered around $0$ as the conclusion, \cite{ST, A, BSX}.
\end{remark}
\begin{remark}
One can show that the function $\psi(t)=<F_t,F>$ is a Morse function and is convex along the flows of the $1$-parameter family $F_t$, if it may be considered over $\mathbb{R}$. Then, the $\mathbb{R}$-orbit $\theta(t)=\exp(it.N)F_0$ satisfies $\theta'(t)=\nabla \ \psi(t)$. Thus $\lambda(t)$ defines the gradient flow line of $\psi(t)$. This implies that the limit 
\begin{equation} 
\lim_{t \to \pm \infty}\ \ (\theta(t)=\exp(it.N)F_0)
\end{equation} 
always exists and should correspond to some critical points of the function $\psi$. In \cite{VRS} the properties of the orbit function $\theta(t)$ have been studied via a moment map $\mu:\mathcal{D} \to \mathfrak{g}$, where $\mathcal{D}$ is the flag variety and $\mathfrak{g}$ is the Lie algebra of the Lie group $G$ acting on $\mathcal{D}$. Then, to a point $F \in \mathcal{D}$ one can associate the invariant 
\begin{equation} 
w_{\mu}(F, u):=\lim_{t \to \infty} \langle \mu(\exp(itu)F),u \rangle , \qquad u \in \mathfrak{g}
\end{equation} 
for a symplectic Riemannian metric $<.,.>$ on $\mathcal{D}$, called Mumford numerical invariant of $F$. By a theorem of Mumford it follows that,
\begin{equation}
    w_{\mu}(gF, gug^{-1})=w_{\mu}(F, u), \qquad g \in G.
\end{equation}
This invariant plays a crucial role in geometric invariant theory on the semi-stability; see \cite{VRS} for details. 
\end{remark}
\begin{remark}
Defining the limit slope filtration in a variation of \'etale or crystalline cohomology with $\mathbb{C}_p$-coefficients is a significant step in geometric Hodge theory. For instance, beginning from a proper smooth map $f:X \to S$, one can associate the $p$-adic (Hodge-Tate) local system 
\begin{equation} 
\mathcal{V}= R^kf_*\mathbb{Q}_p \otimes \mathbb{C}_p=\bigcup_s \left [H_{\textrm{\'et,s}}^k=H_{\textrm{\'et}}^k(X_s,\mathbb{Q}_p) \otimes \mathbb{C}_p \right ].
\end{equation} 
As explained above, we equip this local system with the limit slope filtration corresponding to some choice of nilpotent elements in the Lie algebra of $Gl(V)$. Moreover, it is a canonical semistable filtration.
\end{remark}
%
%
%
\subsection{$p$-adic $SL_2$-orbits}
The nilpotent orbit theorem and the $SL_2$-orbit theorem appear in pairs in the context of asymptotic Hodge theory. We add the statement of the $SL_2$ orbit theorem for the sake of completeness, where at this step having the idea is necessary. The $SL_2$-orbit Theorem states that one can modify a nilpotent orbit to get an $SL_2$-orbit, so it is still asymptotic to the original orbit. The idea is as follows. Let us start with a nilpotent orbit $\exp(N \log \kappa_z) F$, where $F \in \mathcal{D}(\nu)$ (maybe not semistable). By the Jacobson-Morosov theorem, the transformation $X_-=N$ can be paired with transformations $X_+$ and $ Z$ to give a representation,
\begin{equation}
\begin{aligned}
\psi_{p,*}:\mathfrak{sl}_2(\mathbb{C}_p) \to & \mathfrak{g}=Lie(G)\\
\psi_{p,*}(X_+) \in \mathfrak{g}_{-1}, \ \psi_{p,*}(Z) &\in \mathfrak{g}_{0}, \ \psi_{p,*}(X_-=N) \in \mathfrak{g}_{1},
\end{aligned}
\end{equation}
where $X_+,Z,X_-=N$ are $\mathfrak{sl}_2$-triples and $\mathfrak{g}_j:= \{ x \in \mathfrak{g}   \mid  ad (Z)x=j. x \}$. The map $\psi_{p,*}$ can be considered at the level of Lie groups as a homomorphism,
\begin{equation} 
\psi_p:SL(2,\mathbb{C}_p) \to G_{\mathbb{C}_p}.
\end{equation} 
Then, one can define an analytic, horizontal, equivariant embedding as
\begin{equation} 
\tilde{\psi}_p: \mathbb{P}^1(\mathbb{C}_p) \to \mathcal{D}(\nu), \qquad \tilde{\psi}_p(g \circ i)=\psi_p(g) \circ F_0.
\end{equation} 
The function $g$ on a neighborhood of $0$ and defined by
\begin{equation}
\exp(N \log \kappa_z) \circ F=g(\kappa_z)\tilde{\psi}_p(\kappa_z),
\end{equation}
is analytic and is defined into the $p$-adic Lie group $G({\mathbb{C}_p})$. Moreover, $Ad \ g(0)^{-1}(N)$ is the image under $\psi_{p,*}$ of $\left(
\begin{array}{cc}
0 & 1\\
0 & 0  \end{array} \right)$. The proofs are straightforward and analogous to the one in \cite{P} or \cite{Sch}, see also \cite{Vo, KP, GS, D}. We summarize as follows.
\begin{theorem}\label{th:homomorphism}
It is possible to choose a homomorphism of complex Lie groups $\psi_p:SL(2,\mathbb{C}_p) \to G_{\mathbb{C}_p}$, and an analytic, horizontal, equivariant embedding $\tilde{\psi}_p: \mathbb{P}^1(\mathbb{C}_p) \to \mathcal{D}(\nu)$, which is related to $\psi_p$ by $\tilde{\psi}_p(g \circ i)=\psi_p(g) \circ F_0$ and such that $\psi_{p,*}:\mathfrak{sl}_2(\mathbb{C}_p) \to \mathfrak{g}=Lie(G)$ satisfies 
\begin{equation}
\psi_{p,*}(X_+) \in \mathfrak{g}_{-1}, \qquad \psi_{p,*}(Z) \in \mathfrak{g}_{0}, \qquad \psi_{p,*}(X_-=N) \in \mathfrak{g}_{1}
\end{equation}
\noindent
where $X_+,Z,X_-=N$ are $\mathfrak{sl}_2$-triples and an analytic mapping $\ \kappa_z \mapsto g(\log \kappa_z)\ $ of a neighborhood $W$ of $0$ into the $p$-adic Lie group $G_{\mathbb{C}_p}$ such that,
\begin{itemize}
\item[(a)] $\exp(N \log \kappa_z) \circ F=g(\kappa_z)\tilde{\psi}_p(\kappa_z)$;
\item[(b)] $Ad \ g(0)^{-1}(N)$ is the image under $\psi_{p,*}$ of $\left(
\begin{array}{cc}
0 & 1\\
0 & 0  \end{array} \right)$
\end{itemize}
\end{theorem}
\begin{proof} (sketch of \cite{Sch, P})
The point $\exp(N \log \kappa_z ) \circ F$ defines the Hodge-Tate decomposition,
\begin{equation}
H_{\textrm{\'et}}^{k}( \mathbb{C}_p)= \bigoplus_{i=0}^k  H^{i,k-i} (-i)(y)
\end{equation}
of the \'etale cohomology with $\mathbb{C}_p$-coefficients. we have $H^{i,k-i}(-i)(\kappa_z)=h(\kappa_z)H^{i,k-i}(-i))$. The nilpotent orbit $\exp(N\log \kappa_z) \circ F$ provides an $\epsilon$-Hermitian class $(H_{\textrm{\'et}}^{k}( \mathbb{C}_p),B)$ for the representation $H_{\textrm{\'et}}^{k}( \mathbb{C}_p)$. We will choose a basis $(w_1,...,w_r)$ such that they successively generate the Frobenius-Hodge pieces in the slope filtration. It is possible to modify the given basis such that $(w_1(\kappa_z),...,w_r(\kappa_z))$ represents a basis with similar properties and remains orthogonal. Then $w_j(\kappa_z)$ are rational functions of $\kappa_z$. The function $h(\kappa_z)$ is a lifting of the 1-parameter family $\exp(N \log \kappa_z)$. Its logarithmic derivative $A=-2 h^{-1} h'$ takes values in $\mathfrak{g}_{\mathbb{C}_p}$. Setting 
\begin{equation}
A(\kappa_z)=-2h^{-1}h', \qquad F(\kappa_z)=Ad\ h(\kappa_z)^{-1}N, \qquad E(\kappa_z) =-\theta F(\kappa_z)
\end{equation} 
where $\theta$ is a Cartan involution of $\mathfrak{g}$. One is needed to solve the Lax system in analytic coordinates
\begin{equation} \label{eq:lax}
2E'(\kappa_z)=-[A(\kappa_z),E(\kappa(z)], \  2F'(\kappa_z)=[A(\kappa_z),F(\kappa_z)], \  A'(\kappa_z)=-[E(\kappa_z),F(\kappa_z)]
\end{equation}
The solutions to the \eqref{eq:lax} gives the desired $SL_2$-triples. Verifying the properties (a) and (b) in Theorem \ref{th:homomorphism}, it follows that they are the same as the complex case, \cite{Sch, P}.
\end{proof}
We stress the fact that the main reason the whole proof for the complex case extends over the p-adic numbers is that the field $\mathbb{C}_p$ has characteristic zero and is complete. The theory of $SL_2$-triples extends over the Lie algebras over $\mathbb{C}_p$ equally. We refer the reader to \cite{Sch} for the complete proof.

\subsection{p-adic nilpotent orbit theorem in the mixed case} 
The definitions of period domain and period map can be similarly stated for variation of mixed Hodge structure (MHS), \cite{D, Sch, P, KP, GS, Vo}. However, the asymptotic behavior of the period maps and the nilpotent orbits are essentially different from the pure case. Also, the period map can have essential (non-removable) singularities (see \cite{P} for instance), which is never the case in pure HS. We try to discuss this notion for the mixed Hodge structure in p-adic settings, which we introduce in the following definition. The results of this subsection are direct applications of the theory explained in \cite{P} (for complex Hodge structure) to the p-adic case, and Theorem \ref{thm:nilpotent-orbit}. We state them below to fix the ideas.
\begin{definition} (Mixed Hodge (MH) structure in p-adic setting) \label{def:MHT} A $\mathbb{Q}_p$-vector space is said to have a mixed Hodge structure if it is endowed with:
\begin{itemize}
    \item an increasing filtration $P_{\bullet}$ defined over $\mathbb{Q}_p$ (indexed over integers called weight filtration) and,
    \item a decreasing filtration $F^{\bullet}$ defined over $\mathbb{C}_p$ (indexed over rational numbers called slope filtration) 
\end{itemize}
such that the graded pieces $Gr_j^P V$ together with the induced filtration by $F^{\bullet}$ are pure Hodge structure in a sense explained in Section \ref{sec:preliminaries}.
\end{definition}
When $N:V \to V$ is a nilpotent transformation defined over $\mathbb{Q}_p$ then, we call the MH structure $(V,F^{\bullet},P_{\bullet})$ to be $N$-admissible if
\begin{itemize}
\item the limiting slope filtration $F_{\infty}$ exists, and
\item the relative weight filtration $P^{rel}(N,P_{\bullet})$ exists.
\end{itemize}
In this case $(F_{\infty},P^{rel}(N,P_{\bullet}))$ is called the limit mixed Hodge structure. We state the following analog for the Theorem \ref{thm:nilpotent-orbit}. The following theorem is the analogous orbit theorem for mixed Hodge structure period domains.
\begin{theorem} \label{thm:mixed-orbit} (Nilpotent orbit theorem for mixed Hodge-Tate structure) 
Let us assume $\mathcal{V} \to B(r)^*$ is a variation of mixed Hodge structure over the punctured $p$-adic disc, and $N:V \to V$ is a nilpotent transformation on $\mathcal{V}_{s}=V$. Then,
\begin{itemize}
\item $\eta(\kappa_z)=\exp(N \log(\kappa_z))F_{\infty}$ is an $N$-admissible nilpotent orbit, and 
\item for each non-archimedean metric $d$ on $\mathcal{D}^{an,rig}$, we have the following distance estimate 
\begin{equation}
d(\Theta(\log \kappa_z), \eta(\kappa_z)) <C. p^{-k.l/q(n)}r^n, \qquad \forall n, \ (r <<1),
\end{equation}
where $q(n) \to \infty$ as $n \to \infty$.
\end{itemize}
\end{theorem}
\begin{proof}
The first assertion is a consequence of admissibility, and the second follows from the uniform bound in \eqref{eq:fourier-bound}.
\end{proof}

\end{document}